\documentclass[12pt,a4paper,leqno]{amsart}

\usepackage{anysize}
\usepackage[utf8]{inputenc}
\usepackage{amsmath}
\usepackage{amssymb}
\usepackage{array}

\marginsize{3cm}{3cm}{3.5cm}{3.5cm}

\newtheorem{theorem}{Theorem}

\begin{document}

\title{A method for proving Ramanujan series for $1/\pi$}
\author{Jesús Guillera} 
\address{University of Zaragoza, Department of mathematics, 50009 Zaragoza (Spain)}
\email{jguillera@gmail.com}

\maketitle

\begin{abstract}
In a famous paper of $1914$ Ramanujan gave a list of $17$ extraordinary formulas for the number $\pi$. In this paper we explain a general method to prove them, based on an original idea of James Wan and in some own ideas.
\end{abstract}

\section{Introduction}
In his famous paper \cite{Ramanujan} of $1914$ Ramanujan gave a list of $17$ extraordinary formulas for the number $\pi$, which are of the following form
\begin{equation}\label{ramanujan-series}
\sum_{n=0}^{\infty} \frac{\left(\frac12\right)_n\left(\frac1s \right)_n\left(1-\frac1s \right)_n}{(1)_n^3} (a+bn) \, z^n \, = \frac{1}{\pi}, \qquad (c)_{_0}=1, \quad (c)_{_n}=\prod_{j=1}^{n} (c+j-1),
\end{equation}
where $s \in \{2, 3, 4, 6\}$, and $z, b, a$ are algebraic numbers. Four of his formulas are
\begin{align}\label{rama-examples-4}
&\sum_{n=0}^{\infty} \frac{\left(\frac12\right)_n^3}{(1)_n^3} \frac{42n+5}{64^n} = \frac{16}{\pi}, \\
&\sum_{n=0}^{\infty} \frac{\left(\frac12\right)_n\left(\frac13 \right)_n\left(\frac23 \right)_n}{(1)_n^3}
(33n+4) \left( \frac{4}{125} \right)^{\!n} = \frac{15 \sqrt 3}{2\pi}, \\
&\sum_{n=0}^{\infty} \frac{\left(\frac12\right)_n\left(\frac14 \right)_n\left(\frac34 \right)_n}{(1)_n^3}
\frac{26390n+1103}{99^{4n+2}} = \frac{\sqrt 2}{4\pi}, \\
&\sum_{n=0}^{\infty} \frac{\left(\frac12\right)_n\left(\frac16 \right)_n\left(\frac56 \right)_n}{(1)_n^3} (63n+8) \left( \frac{-4}{5} \right)^{\! 3n} = \frac{5\sqrt{15}}{\pi},
\end{align}
corresponding to $s=2, 3, 4, 6$ respectively. However Ramanujan wrote few details of his proofs and the first rigorous deductions were made by the Borwein brothers in $1985$, see \cite{BorweinBorwein}. Other general proofs based on the modular theory are for example in \cite{Baruah-Berndt-1, Baruah-Berndt-2, BorweinBorwein, Baruah-Berndt-Chan, Chan-Liaw-Tan, Zudilin}, and we believe that the method that we use in this paper is an interesting alternative one. Other kind of proofs are in \cite{Gui-Zud-translation, Coo-Zu, Cooper-book, Gui-WZ-pairs, Gui-rama-termina}, and two beautiful surveys are \cite{Baruah-Berndt-Chan} and \cite{Zudilin}.

\section{A method for proving Ramanujan series for $1/\pi$}

We will use the following notation:
\[
F_s(\alpha)={}_2F_1\biggl(\begin{matrix} \frac1s, \, 1-\frac1s \\ 1 \end{matrix}\biggm| \alpha \biggr), \quad G_s(\alpha)=\alpha \frac{dF_s(\alpha)}{d\alpha},
\]
and the following version of the Legendre's relation:
\begin{equation}\label{Legendre}
\alpha F_s(\alpha) G_s(\beta)+\beta F_s(\beta) G_s(\alpha) = \frac{1}{\pi} \, \sin \frac{\pi}{s}, \qquad \beta = 1-\alpha.
\end{equation}
We will show that this relation explains why $\pi$ appears in the Ramanujan series. The other ingredients we need to prove the series are: a transformation of modular origin and the known Clausen's identity
\begin{equation}\label{factorization}
\sum_{n=0}^{\infty} \frac{\left(\frac12\right)_n\left(\frac1s \right)_n\left(1-\frac1s \right)_n}{(1)_n^3} z^n =
F_{_s} (\alpha) F_{_s}(\alpha), \qquad z=4 \alpha (1-\alpha).
\end{equation}
Our method consist in  a variant of a James Wan's original idea \cite{Wan}. We explained it in \cite{Gui-fam-rama-orr} and \cite{Gui-more-rama-orr} proving some new Ramanujan-Orr series for $1/\pi$. In this paper we show how to apply it to prove Ramanujan series in a simple way when we know the required transformation, or what is equivalent: when we know the required modular equation, because the multiplier is given by the formula (\ref{m-alpha-beta}). We explain our technique with two examples: one corresponding to a series of positive terms and the other one to an alternating series.

\subsection{Example for series of positive terms}
We reprove below the following series for $1/\pi$ of level $\ell=2$ ($s=4$) due to Ramanujan:
\begin{equation}\label{rama-81}
\sum_{n=0}^{\infty} \frac{\left(\frac12\right)_n\left(\frac14\right)_n\left(\frac34\right)_n}{(1)_n^3} \frac{1}{3^{4n}} (10n+1) = \frac{9 \sqrt 2}{4 \pi}.
\end{equation}
We begin with the following theorem:
\begin{theorem}
For $s=4$ (level $\ell=2$), we have
\begin{align}
F_{_4} (\alpha_0) &=\frac{\sqrt{5}}{5} F_{_4} (\beta_0), \label{F-relation1} \\
G_{_4} (\alpha_0) &= \frac{16 \sqrt{5}-36}{5} F_{_4} (\beta_0)+\frac{161 \sqrt{5}-360}{5} G_{_4} (\beta_0), \label{G-relation1}
\end{align}
where
\[
\alpha_0 = \frac12-\frac{2\sqrt 5}{9}, \quad \beta_0 = 1 - \alpha_0 = \frac12+\frac{2\sqrt 5}{9}.
\]
\end{theorem}

\begin{proof}
In \cite[page  608]{Cooper-book} we see the following transformation of level 2 and degree $1/d$ with $d=5$: 
\begin{equation}\label{F-level4-degree5}
F_4 (\alpha) = m F_4(\beta)
\end{equation}
where
\[
\alpha=\frac{64x^5(1+x)}{(1+4x^2)(1-2x-4x^2)^2}, \quad \beta = \frac{64x(1+x)^5}{(1+4x^2)(1+22x-4x^2)^2}, 
\]
and
\[
m=\frac{\sqrt{1-2x-4x^2}}{\sqrt{1+22x-4x^2}}.
\]
If we take $\beta=1-\alpha$, then we get the following solution
\[
x_0=\frac{\sqrt 5 -2}{2}, \quad \alpha_0=\frac12-\frac{2\sqrt 5}{9}, \quad \beta_0=\frac12+\frac{2\sqrt 5}{9}, \quad m_0=\frac{1}{\sqrt 5}.
\]
In addition, we have
\[
\alpha'_0 = \frac{\sqrt 5 + 2}{27}, \quad \beta'_0 = \frac{\sqrt 5 + 2}{27}, \quad m'_0 = \frac{-8 \sqrt{5} - 16}{15}.
\]
Differentiating (\ref{F-level4-degree5}) with respect to $\alpha$ we have
\begin{equation}\label{G-level4-degree5}
G_{_4}(\alpha) = \alpha \, \frac{m'}{\alpha'} F_{_4}(\beta) + \alpha \, \frac{m}{\beta} \, \frac{\beta'}{\alpha'} \, G_{_4}(\beta),
\end{equation}
where the $'$ stands for the derivative with respect to $x$. Finally, substituting $x=x_0$, in (\ref{F-level4-degree5}) and (\ref{G-level4-degree5}) we arrive at the results stated by the theorem.
\end{proof}
We are ready to prove (\ref{rama-81}). 
\begin{proof}
Applying to both sides of (\ref{factorization}) with $s=4$ the operator 
\[
\left. \frac29+\frac{20}{9} z \frac{d}{dz} \right|_{z=z_0} = \left. \frac29+\frac{20}{9} \frac{z}{z'} \frac{d}{d\alpha} \right|_{\alpha=\alpha_0}
\]
where here the $'$ means the derivative with respect to $\alpha$, we obtain
\begin{multline}
\sum_{n=0}^{\infty} \frac{\left(\frac12\right)_n\left(\frac14\right)_n\left(\frac34\right)}{(1)_n^3} \frac{1}{3^{4n}} \left(\frac{20}{9}n+\frac29 \right) \\ 
= \frac29 F_{_4}(\alpha_0) F_{_4}(\alpha_0) + \left( \frac{\sqrt 5}{2} + \frac{10}{9} \right) F_{_4}(\alpha_0) G_{_4}(\alpha_0) + \left( \frac{\sqrt 5}{2} + \frac{10}{9} \right) G_{_4}(\alpha_0) F_{_4}(\alpha_0).
\end{multline}
Observe that we have intentionally repeated two equal terms without simplifying the sum. Then, if we use the relations (\ref{F-relation1}) and (\ref{G-relation1}) to replace one factor $F_4(\alpha_0)$ of the two first terms and to replace $G_4(\alpha_0)$ in the last term, we arrive at
\[
\alpha_0 F_{_4(}\alpha_0) G_{_4}(\beta_0) + \beta_0 F_{_4}(\beta_0) G_{_4}(\alpha_0),
\]
which in view of (\ref{Legendre}) is equal to
\[
\frac{1}{\pi} \sin \frac{\pi}{4} = \frac{\sqrt 2}{2 \pi},
\]
and we are done.
\end{proof}

\subsection{Example for alternating series}
Here we prove with our method the following alternating series of level $\ell=2$ due to Ramanujan:
\begin{equation}\label{rama-m48}
\sum_{n=0}^{\infty} \frac{\left(\frac12\right)_n\left(\frac14\right)_n\left(\frac34\right)_n}{(1)_n^3} \left(\frac{-1}{48}\right)^{\!n} (28n+3) = \frac{16 \sqrt 3}{3 \pi}.
\end{equation}

\begin{theorem}
For $s=4$ (level $\ell=2$), we have
\begin{equation}\label{F-relation2}
F_{_4} (\alpha_0) = \frac{(3 + i)\sqrt 2}{10} F_{_4} (\beta_0), 
\end{equation}
and
\begin{multline}\label{G-relation2}
G_{_4} (\alpha_0) = \left( \frac{-63}{20} \sqrt{2} +\frac{9 \sqrt 6}{5} \right) F_{_4} (\beta_0)
\\ + \left[ \left(  \frac{-291}{10} \sqrt 2 + \frac{84}{5} \sqrt{6} \right) + 
\left( \frac{-28}{5} \sqrt{6} + \frac{97}{10} \sqrt{2} \right) i \right] G_{_4}(\beta_0),
\end{multline}
where
\[
\alpha_0=\frac12-\frac{7\sqrt 3}{24}, \quad \beta_0=\frac12+\frac{7\sqrt 3}{24}.
\]
\end{theorem}

\begin{proof}
We use again the transformation of degree $1/d$ with $d=5$:
\begin{equation}\label{F-same-transform}
F_4 (\alpha) = m F_4(\beta),
\end{equation}
where
\[
\alpha=\frac{64x^5(1+x)}{(1+4x^2)(1-2x-4x^2)^2}, \quad \beta = \frac{64x(1+x)^5}{(1+4x^2)(1+22x-4x^2)^2}, 
\]
and
\[
m=\frac{\sqrt{1-2x-4x^2}}{\sqrt{1+22x-4x^2}}.
\]
Another solution of $\beta=1-\alpha$ is the following one:
\[
x_0=\frac{2\sqrt 3 - 3}{4}-\frac{2-\sqrt 3}{4} i, \quad \alpha_0=\frac12-\frac{7\sqrt 3}{24}, \quad \beta_0=\frac12+\frac{7\sqrt 3}{24}, \quad m_0=\frac{(3+i)\sqrt 2}{10}.
\]
It is interesting to note that $|m_0|=1/\sqrt{5}$. We also get
\begin{align}
m'_0 &= \left( \frac{-27}{40} \sqrt{2} - \frac{69}{200} \sqrt{6} \right)-\left( \frac{33}{200} \sqrt{6} + \frac{9}{40} \sqrt{2} \right) i,  \\
\alpha'_0 &= \left( \frac{-23}{240}-\frac{\sqrt{3}}{16} \right) - \left( \frac{\sqrt{3}}{48} + \frac{11}{240}\right) i, \\
\beta'_0 &= \left( \frac{-5}{48} -\frac{\sqrt{3}}{16}\right) + \left( \frac{1}{48} + \frac{\sqrt{3}}{48} \right) i.
\end{align}
Differentiating (\ref{F-same-transform}) with respect to $\alpha$ we have
\begin{equation}\label{G-same-transform}
G_{_4}(\alpha) = \alpha \, \frac{m'}{\alpha'} F_{_4}(\beta) + \alpha \, \frac{m}{\beta} \, \frac{\beta'}{\alpha'} \, G_{_4}(\beta),
\end{equation}
where the $'$ stands for the derivative with respect to $x$. Finally, substituting $x=x_0$, in (\ref{F-same-transform}) and (\ref{G-same-transform}) we arrive at the results stated by the theorem.
\end{proof}
We are ready to prove (\ref{rama-m48}). 
\begin{proof}
Applying to both sides of (\ref{factorization}) with $s=4$ the operator 
\[
\left. \frac{3}{32}\sqrt{6}+\frac{28}{32} \sqrt{6}z \frac{d}{dz} \right|_{z=z_0} = \left. \frac{3}{32}\sqrt{6}+\frac{28}{32}\sqrt{6} \frac{z}{z'} \frac{d}{d\alpha} \right|_{\alpha=\alpha_0}
\]
where here the $'$ means the derivative with respect to $\alpha$, we obtain
\begin{multline}
\sum_{n=0}^{\infty} \frac{\left(\frac12\right)_n\left(\frac14\right)_n\left(\frac34\right)}{(1)_n^3} \left(\frac{-1}{48}\right)^n \left(\frac{28}{32} \sqrt{6} n+\frac{3}{32} \sqrt{6} \right) 
= \frac{3\sqrt 6}{32} F_{_4}(\alpha_0) F_{_4}(\alpha_0) \\ + (1-C i) \left( \frac{3\sqrt 2}{4} + \frac{7}{16} \sqrt{6} \right) F_{_4}(\alpha_0) G_{_4}(\alpha_0) + (1+C i) \left( \frac{3\sqrt 2}{4} + \frac{7}{16} \sqrt{6} \right) G_{_4}(\alpha_0) F_{_4}(\alpha_0).
\end{multline}
Observe that we have intentionally introduced the factors $(1-C i)$ and $(1+C i)$, and that the expression holds for all values of $C$ because the terms with $C$ cancels. Then, we use the relations (\ref{F-relation2}) and (\ref{G-relation2}) to replace one factor $F_4(\alpha_0)$ of the two first terms and to replace $G_4(\alpha_0)$ in the last term. Finally, if we choose $C=1/3$,  we arrive at
\[
\alpha_0 F_{_4(}\alpha_0) G_{_4}(\beta_0) + \beta_0 F_{_4}(\beta_0) G_{_4}(\alpha_0),
\]
which in view of (\ref{Legendre}) is equal to
\[
\frac{1}{\pi} \sin \frac{\pi}{4} = \frac{\sqrt 2}{2 \pi},
\]
and we are done.
\end{proof}

\section{Explicit formulas}

We already know that $z=4\alpha (1-\alpha)$. Here, we give explicit formulas for $b$ and $a$, and also for the modular variable $q$.

\subsection{On modular equations and multipliers}

Transformations of modular origin can be used to prove Ramanujan-type series for $1/\pi$. Those proved in \cite{Cooper-book} and \cite{Coo-Zu} are written like the one used in this paper. For the great quantity of them given by Ramanujan, see \cite[Chapters 19, 20]{Be-vol-3} and \cite[Chapters 33, 36]{Be-vol-5}. 
For the modular equations $P_s(\alpha, \beta)=0$ corresponding to $s=2,3,4,6$ (levels $\ell=4,3,2,1$), the multiplier is given by the following formula:
\begin{equation}\label{m-alpha-beta}
m_s(\alpha, \beta) = \frac{1}{\sqrt{d}} \left( \frac{\beta(1-\beta)}{\alpha(1-\alpha)} \frac{d \alpha}{d \beta} \right)^{1/2},
\end{equation}
where $1/d$ is the degree of the modular equation. Hence the associated transformation of level $\ell$ and degree $1/d$ reads
\[
F_{_s}(\alpha) = m_s(\alpha, \beta) F_{_s}(\beta), \qquad P_s(\alpha, \beta) = 0.
\]
You can see a proof of (\ref{m-alpha-beta}) for the case $s=2$ (level $\ell=4$) in \cite[Entry 24 (vi)]{Be-vol-3}, but a similar proof can be given for the four hypergeometric levels. 

\subsection{Explicit formulas}
To get the explicit formulas for $a$ and $b$ we will use our method and the identity 
\begin{equation}\label{dalpha-dbeta-m}
\frac{\beta'_0}{\alpha'_0} = \frac{1}{d \, m_0^2},
\end{equation}
which comes from (\ref{m-alpha-beta}). First, observe that the level is related to $s$ in the following way:
\begin{equation}\label{level-s}
\ell = 4 \sin^2 \frac{\pi}{s}.
\end{equation}
It is not a coincidence \cite[eq. 27]{Gui-Zud-translation}. Using our method for the case of general series and cases $z_0>0$ or $z_0<0$, that is, applying the operator
\[
a+ \left. b z \frac{d}{dz} \right|_{z_0},
\]
to both sides of (\ref{factorization}), using the substitutions
\[
F_s(\alpha_0) = m F_s(\beta_0), \qquad 
G_s(\alpha_0)=\alpha_0 \frac{m'_0}{\alpha'_0} F_s(\beta_0) + \frac{\alpha_0}{d \, m_0 \, \beta_0} G_s(\beta_0),
\]
in the way that we have explained in the examples, and equating the coefficients of $F(\alpha_0)F(\beta_0)$, $F(\alpha_0)G(\beta_0)$ and $F(\beta_0)G(\alpha_0)$ to $0$, $\alpha_0$ and $\beta_0$ respectively, we deduce the following explicit formulas:
\begin{equation}\label{b-alternating-terms-gen}
b=(1-2\alpha_0) \, \frac{{\rm Re}(m_0)}{\sin \frac{\pi}{s}} \, d, 
\end{equation}
and
\begin{equation}\label{a-alternating-terms-gen}
a=-(1+Ci) \, \frac{\alpha_0 \beta_0}{\alpha'_0} \, \frac{m'_0}{m_0} \, \frac{b}{1-2\alpha_0}, \qquad C=\frac{{\rm Im}(m_0)}{{\rm Re}(m_0)}.
\end{equation}
For the case $z>0$, we have $m_0=1/\sqrt{d}$. Hence, we can write the above formulas in the way
\begin{equation}\label{b-positive-terms}
b=2 (1-2\alpha_0) \, \sqrt{\frac{d}{\ell}}, \qquad a=-2\alpha_0 \beta_0 \frac{m'_0}{\alpha'_0} \, \frac{d}{\sqrt{\ell}}.
\end{equation}
For the case $z<0$ (alternating series), we have observed experimentally that
\begin{equation}\label{conj-m0}
m_0 \, {\overset{?}=} \, \frac{\sqrt{4d-\ell}}{2d}+\frac{\sqrt{\ell}}{2d} \, i, 
\end{equation}
Hence, assuming it, we can write (\ref{b-alternating-terms-gen}) and (\ref{a-alternating-terms-gen}) in the following form:
\begin{equation}\label{b-alternating-terms}
b=2 \, (1-2\alpha_0) \sqrt{\frac{d}{\ell}-\frac14}, \qquad a=-2\alpha_0 \beta_0 \frac{m'_0}{\alpha'_0} \, \frac{d}{\sqrt{\ell}}.
\end{equation}
Finally, we relate the modular variable $q$ with $d$ and $\ell$ assuming (\ref{conj-m0}). Let $q=e^{-2\pi\sqrt{r}}$ and $q=-e^{-2\pi\sqrt{r}}$, the modular variable corresponding to the cases $z>0$ and $z<0$, respectively. From the known formula
\[
4r=\frac{b^2}{1-z}=\frac{b^2}{(1-2\alpha)^2},
\]
we deduce that
\[
r=\frac{d}{4\sin^2 \frac{\pi}{s}} = \frac{d}{\ell}, \quad r=\frac{d}{4\sin^2 \frac{\pi}{s}} - \frac14= \frac{d}{\ell} - \frac14,
\]
for the cases $z>0$ and $z<0$, respectively. Hence,
\begin{equation}\label{q}
q = e^{-2\pi \sqrt{\frac{d}{\ell}}}, \qquad  q = - e^{-2\pi \sqrt{\frac{d}{\ell}-\frac14}},  
\end{equation}
for all the series of positive terms and for all the alternating series respectively.

\subsection{An experimental test}

The test which consist of evaluating numerically 
\begin{equation}\label{test-degree}
\frac{F_s(\alpha_0)}{F_s(\beta_0)}=m_0, \quad \beta_0=1-\alpha_0,
\end{equation}
has been very useful in discovering that $|m_0^2|$ but not $m_0^2$ (algebraic) is a positive rational number, and that $|m_0^2|=1/d$. For example, for the series of level $\ell=3$ ($s=3$) and $z_0=-1/500^2$, as $z_0=4\alpha_0 (1-\alpha_0)$ we obtain 
\[
\alpha_0=\frac12 - \frac{53\sqrt{89}}{1000}, \quad \beta_0=1-\alpha_0 = \frac12 + \frac{53\sqrt{89}}{1000},
\]
and evaluating numerically (\ref{test-degree}), we get with an approximation of $20$ digits that
\[
\left| \frac{F_s(\alpha_0)}{F_s(\beta_0)} \right| \approx  0.20851441405707476267,
\]
which we identify as $1/\sqrt{23}$. Hence, for proving with our method that alternating series, we need a transformation of degree $1/d$ with $d=23$ for the level $\ell=3$, and with such a transformation we can prove it rigorously. See the tables at the end of the paper.

\section*{Acknowledgements}
I am grateful to Shaun Cooper for very interesting questions related to the method used.

\begin{table}[p]
    \begin{tabular}{|c c c c | c c c c|}
        \hline &&&&&&& \\ [-1ex]
        $d$ & $a$ & $b$ & $z<0$ & $d$ & $a$ & $b$ & $z>0$ \\ &&&&&&& \\ [-1ex]
        \hline \hline &&&&&&& \\ 
        $3$ & $\frac12$ & $2$ & $-1$ &
        $3$ & $\frac14$ & $\frac64$ & $\frac14$ \\ &&&&&&& \\
        $5$ & $\frac{1}{2\sqrt2}$ & $\frac{6}{2\sqrt2}$ & $-\frac{1}{8}$ &
        $7$ & $\frac{5}{16}$ & $\frac{42}{16}$ & $\frac{1}{64}$ \\ &&&&&&& \\
        \hline
    \end{tabular}
    \vskip 0.25cm
    \caption{Rational Ramanujan-type series of $\ell=4$ for $1/\pi$}
\end{table}

\begin{table}[b]
    \begin{tabular}{|c c c c | c c c c|}
        \hline &&&&&&& \\ [-1ex]
        $d$ & $a$ & $b$ & $z<0$ & $d$ & $a$ & $b$ & $z>0$ \\ &&&&&&& \\ [-1ex]
        \hline \hline &&&&&&& \\
        $3$ & $\frac{3}{8}$ & $\frac{20}{8}$ & $-\frac{1}{4}$ &
        $2$ & $\frac{2}{9}$ & $\frac{14}{9}$ & $\frac{32}{81}$ \\ &&&&&&& \\
        $4$ & $\frac{8}{9\sqrt7}$ & $\frac{65}{9\sqrt7}$ & $-\frac{16^2}{63^2}$ &
        $3$ & $\frac{1}{2\sqrt3}$ & $\frac{8}{2\sqrt3}$ & $\frac{1}{9}$ \\ &&&&&&& \\
        $5$ & $\frac{3\sqrt3}{16}$ & $\frac{28\sqrt3}{16}$ & $-\frac{1}{48}$ &
        $5$ & $\frac{4}{9\sqrt2}$ & $\frac{40}{9\sqrt2}$ & $\frac{1}{81}$ \\ &&&&&&& \\
        $7$ & $\frac{23}{72}$ & $\frac{260}{72}$ & $-\frac{1}{18^2}$ &
        $9$ & $\frac{27}{49\sqrt3}$ & $\frac{360}{49\sqrt3}$ & $\frac{1}{7^4}$ \\ &&&&&&& \\
        $13$ & $\frac{41\sqrt5}{288}$ & $\frac{644\sqrt5}{288}$ & $-\frac{1}{5 \cdot 72^2}$ &
        $11$ & $\frac{19}{18\sqrt{11}}$ & $\frac{280}{18\sqrt{11}}$ & $\frac{1}{99^2}$ \\ &&&&&&& \\
        $19$ & $\frac{1123}{3528}$ & $\frac{21460}{3528}$ & $-\frac{1}{882^2}$ &
        $29$ & $\frac{4412}{9801\sqrt2}$ & $\frac{105560}{9801\sqrt2}$ & $\frac{1}{99^4}$ \\ &&&&&&& \\
        \hline
    \end{tabular}
    \vskip 0.25cm
    \caption{Rational Ramanujan-type series of $\ell=2$ for $1/\pi$}
\end{table}

\begin{table}[ht]
    \begin{tabular}{|c c c c | c c c c|}
        \hline &&&&&&& \\ [-1ex]
        $d$ & $a$ & $b$ & $z<0$ & $d$ & $a$ & $b$ & $z>0$ \\ &&&&&&& \\ [-1ex]
        \hline \hline &&&&&&& \\
        $3$ & $\frac{\sqrt{3}}{4}$ & $\frac{5\sqrt{3}}{4}$ & $-\frac{9}{16}$ &
        $2$ & $\frac{1}{3\sqrt{3}}$ & $\frac{6}{3\sqrt{3}}$ & $\frac{1}{2}$ \\ &&&&&&& \\
        $5$ & $\frac{7}{12\sqrt3}$ & $\frac{51}{12\sqrt3}$ & $-\frac{1}{16}$ &
        $4$ & $\frac{8}{27}$ & $\frac{60}{27}$ & $\frac{2}{27}$ \\ &&&&&&& \\
        $7$ & $\frac{\sqrt{15}}{12}$ & $\frac{9\sqrt{15}}{12}$ & $-\frac{1}{80}$ &
        $5$ & $\frac{8}{15\sqrt3}$ & $\frac{66}{15\sqrt3}$ & $\frac{4}{125}$ \\ &&&&&&& \\
        $11$ & $\frac{106}{192\sqrt3}$ & $\frac{1230}{192\sqrt3}$ & $-\frac{1}{2^{10}}$
        &&&& \\ &&&&&&& \\
        $13$ & $\frac{26\sqrt7}{216}$ & $\frac{330\sqrt7}{216}$ & $-\frac{1}{3024}$
        &&&& \\ &&&&&&& \\
        $23$ & $\frac{827}{1500\sqrt3}$ & $\frac{14151}{1500\sqrt3}$ & $-\frac{1}{500^2}$
        &&&& \\ &&&&&&& \\
        \hline
    \end{tabular}
    \vskip 0.25cm
    \caption{Rational Ramanujan-type series of $\ell=3$ for $1/\pi$}
\end{table}

\begin{table}[ht]
    \begin{tabular}{|c c c c | c c c c|}
        \hline &&&&&&& \\ [-1ex]
        $d$ & $a$ & $b$ & $z<0$ & $d$ & $a$ & $b$ & $z>0$ \\ &&&&&&& \\ [-1ex]
        \hline \hline &&&&&&& \\
        $2$ & $\frac{8}{5\sqrt{15}}$ & $\frac{63}{5\sqrt{15}}$ & $-\frac{4^3}{5^3}$ &
        $2$ & $\frac{3}{5\sqrt5}$ & $\frac{28}{5\sqrt5}$ & $\frac{3^3}{5^3}$ \\ &&&&&&& \\
        $3$ & $\frac{15}{32\sqrt2}$ & $\frac{154}{32\sqrt2}$ & $-\frac{3^3}{8^3}$ &
        $3$ & $\frac{6}{5\sqrt{15}}$ & $\frac{66}{5\sqrt{15}}$ & $\frac{4}{5^3}$ \\ &&&&&&& \\
        $5$ & $\frac{25}{32\sqrt6}$ & $\frac{342}{32\sqrt6}$ & $-\frac{1}{8^3}$ &
        $4$ & $\frac{20}{11\sqrt{33}}$ & $\frac{252}{11\sqrt{33}}$ & $\frac{2^3}{11^3}$ \\ &&&&&&& \\
        $7$ & $\frac{279}{160\sqrt{30}}$ & $\frac{4554}{160\sqrt{30}}$ & $-\frac{9}{40^3}$ &
        $7$ & $\frac{144\sqrt3}{85\sqrt{85}}$ & $\frac{2394\sqrt{3}}{85\sqrt{85}}$ & $\frac{4^3}{85^3}$ \\ &&&&&&& \\
        $11$ & $\frac{526\sqrt{15}}{80^2}$ & $\frac{10836 \sqrt{15}}{80^2}$ & $-\frac{1}{80^3}$ &&&& \\ &&&&&&& \\
        $17$ & $\frac{10177\sqrt{330}}{3 \cdot 440^2}$ & $\frac{261702\sqrt{330}}{3\cdot 440^2}$ & $-\frac{1}{440^3}$ &&&& \\ &&&&&&& \\
        $41$ & $\frac{27182818\sqrt{10005}}{3 \cdot 53360^2}$ & $\frac{1090280268\sqrt{10005}}{3\cdot 53360^2}$ & $-\frac{1}{53360^3}$ &&&& \\ &&&&&&& \\
        \hline
    \end{tabular}
    \vskip 0.25cm
    \caption{Rational Ramanujan-type series of $\ell=1$ for $1/\pi$}
\end{table}

\end{document}